\documentclass[reqno,11pt]{amsart}
\usepackage{amsmath,amssymb}
\usepackage[T1]{fontenc}
\usepackage[latin1]{inputenc}
\def\BB{{\mathcal B}}
\def\CC{{\mathcal C}}
\def\FF{{\mathcal F}}
\def\GG{{\mathcal G}}
\def\qed
 {\ifhmode\unskip\nobreak\hfill$\Box$\medskip\fi
 \ifmmode\eqno{\Box}\fi}
\def\HH{{\mathcal H}}

\DeclareMathOperator{\ex}{ex}
\def\exrlin{{\ex_r^{{\rm lin}}}}
\def\extrilin{{\ex_3^{{\rm lin}}}}

\newtheorem{thm}{Theorem}
\newtheorem{lem}[thm]{Lemma}
\newtheorem{cor}[thm]{Corollary}
\newtheorem{prop}[thm]{Proposition}

\begin{document}
\title[On 3-uniform hypergraphs without a cycle of a given length]
{On 3-uniform hypergraphs \\ without a cycle of a given length}

\author[Z.~F\"uredi]{Zolt\'an F\"uredi}
\address{Alfr\'ed R\'enyi Institute of Mathematics of the Hungarian Academy of Sciences, Budapest, P.O.Box 127, Hungary, H-1364}
\email{z-furedi@math.uiuc.edu}
\author[L.~\"Ozkahya]{Lale \"Ozkahya}
\address{Department of Computer Engineering, Hacettepe University, Beytepe, Ankara, Turkey}
\email{ozkahya@hacettepe.edu.tr}
\thanks{The research of the first author is supported in part by the Hungarian National Science Foundation
 OTKA 104343, by the European Research Council Advanced Investigators Grant 26719
 and by the Simons Foundation grant 317487.
\newline\indent
This work was done while the first author visited the Department of Mathematics and Computer Science, Emory University, Atlanta, GA, USA.
\newline\indent
A major revision of the paper was done during a visit of the first author
 to the Institut Mittag-Leffler (Djursholm, Sweden).
    }

\begin{abstract}
We study the maximum number of hyperedges in a 3-uniform hypergraph on $n$ vertices
that does not contain a Berge cycle of a given length $\ell$.
In particular we prove that the upper bound for $C_{2k+1}$-free
hypergraphs is of the order
    $O(k^2n^{1+1/k})$,
improving the upper bound of Gy{\H o}ri and Lemons~\cite{gy-nat-3unif} by a factor of $\Theta(k^2)$.
Similar bounds are shown for linear hypergraphs.
\end{abstract}

\maketitle

\section{A generalization of the Tur\'an problem}

Counting substructures is a central topic of extremal combinatorics.
Given two (hyper)graphs $G$ and $H$ let $N(G\, ;H)$ denote the number of subgraphs of $G$ isomorphic to $H$.
(Usually we consider a labelled host graph $G$).
Note that $N(G\, ;K_2)=e(G)$, the number of edges of $G$.
More generally, $N(\GG\, ; H)$ is the maximum of $N(G\, ;H)$ where $G\in \GG$, a class of graphs.
In most cases, in Tur\'an type problems, $\GG$ is a set of $n$-vertex $\FF$-free graphs, where
 $\FF$ is a collection of forbidden subgraphs.
This maximum is denoted by $N(n,\FF\, ;H)$.
So $N(n, \FF\, ; H)$ is the maximum number of copies of $H$ in an $\FF$-free graph on $n$ vertices.
The Tur\'an number $\ex(n,\FF)$ is defined as $N(n,\FF\, ;K_2)$.
Let  $\ex(m,n,\FF)$ be the maximum number edges in a bipartite graph with parts
 of order $m$ and $n$ vertices that do not contain any member of $\FF$.
$\CC_{\ell}$ is the family of all cycles of length at most $\ell$.
For any graph $G$ and any vertex $x$,
we let $t(G)$ and $t(x)$ denote the number of triangles in $G$ and
the number of triangles containing $x$, respectively.
Let  $t_{\ell}(n):= N(n,C_{\ell}\, ; K_3)$.

Our starting point is the Bondy-Simonovits~\cite{Bondy-Simon} theorem, $\ex(n,C_{2k})\leq 100k n^{1+ 1/k}$.
Recall two contemporary versions  due to Pikhurko~\cite{pik-2011}, Bukh and Z.~Jiang~\cite{bukh}, respectively,
 and a classical result by K\H{o}v\'ari, T. S\'os, and Tur\'an~\cite{KovSosTur}.
For all $k\ge 2$ and $n\ge 1$, we have
\begin{eqnarray}
  \ex(n,C_{2k})&\leq& (k-1)n^{1+1/k}+16(k-1)n, \label{pik}\\
   \ex(n,C_{2k})&\leq& 80\sqrt{k \log k} n^{1+1/k} + 10k^2n,  \label{eq2}\\
  \ex(n,n, C_4) &\leq& n^{3/2}+ 2n. \label{eq3}
  \end{eqnarray}

Erd{\H o}s~\cite{erdos} conjectured that a triangle-free graph on $n$ vertices
can have at most $(n/5)^5$ five cycles and that equality holds for the
blown-up $C_5$ if $5|n$.
Gy{\H o}ri~\cite{gyC5} showed that a triangle-free graph on $n$
vertices contains at most $c(n/5)^5$ copies of $C_5$, where $c < 1.03$.
Grzesik~\cite{Grz}, and independently, Hatami et al.~\cite{Razb-pentagon}
confirmed that Erd{\H o}s' conjecture is true by using Razborov's method of flag algebras, i.e.,
  $N(n,C_3\, ;C_5)\leq (n/5)^5$.

Bollob\'as and Gy{\H o}ri~\cite{bol-gy-2008} asked a related question:
how many triangles can a graph have if it does not contain a $C_5$.
They obtained the upper bound $t_5(n)\le (1+o(1))(5/4) n^{3/2}$ which yields the correct order of magnitude.

Later, Gy{\H o}ri and Li~\cite{gy-li-2009} provided bounds on $t_{2k+1}(n)$.
\begin{equation}\label{eq4}
\binom{k}{2} \ex\left(\frac{n}{k+1}, \frac{n}{k+1}, \CC_{2k}\right) \leq
t_{2k+1}(n)\le \frac{(2k-1)(16k-2)}{3}\ex(n,C_{2k}).
\end{equation}

In Section~\ref{t2k+1} we improve the upper bound by a factor of $\Omega(k)$.
\begin{thm}\label{ct-tri} For $k\ge 2$,
\begin{eqnarray}\label{eq5}
t_{2k+1}(n)&:= & N(n,C_{2k+1}\, ; K_3)\le 9(k-1) \ex\left(\left\lceil \frac{n}{3} \right\rceil,
\left\lceil  \frac{n}{3} \right\rceil, C_{2k}\right),\\
t_{2k}(n) &\le & \frac{2k - 3}{3}\ex(n, C_{2k}).
 \label{eq6}
\end{eqnarray}
\end{thm}

The inequalities (\ref{pik}), (\ref{eq3}) and (\ref{eq5}) give
 $t_{2k+1}(n)\leq
          9(k-1)^2\left((2/3)n\right)^{1+1/k}+O(n)$
for $k\geq 3$ and $t_{5}(n)\leq  \sqrt{3}n^{3/2} + O(n)$. This latter one is not better than
 the Bollob\'as-Gy\H ori bound.
However, our constant factor in Theorem~\ref{ct-tri} is the best possible in the following sense.
It is widely believed that that the Tur\'an numbers in the above statements are 'smooth', i.e.,
there are constants $a_k, b_k$ depending only on $k$ such that
$\ex(n,n,C_{2k}) = (a_k+o(1)) n^{1 + 1/k}$ and
$\ex(n,n,\CC_{2k}) = (b_k+o(1)) n^{1 + 1/k}$.
If these are indeed true then the ratio of the upper bound in (\ref{eq5}) and the lower bound in (\ref{eq4})
 is bounded by a constant factor of $O(a_k/b_k)$.
It is also believed that the sequence  $a_k/b_k$ is bounded (as $k\to \infty$), so further essential improvement is probably not possible.

Since the first version of this manuscript (2011)
 Alon and Shikhelman~\cite{alon-shik} improved the upper bound in Theorem~\ref{ct-tri} by a constant factor
to  $(16/3)(k-1) \ex(\lceil n/2\rceil, C_{2k})$ and showed that $t_5(n)\le (1+o(1))(\sqrt{3}/2) n^{3/2}$.
Nevertheless, we include our proof in Section~\ref{t2k+1} for completeness, and because we use
 Theorem~\ref{ct-tri} in our main result in the next section.

\section{Berge cycles}

A {\it Berge cycle} of length $k$ is a family of
distinct hyperedges $H_0,$ $\dots,$ $H_{k-1}$ such that
there are distinct vertices $v_0,\dots,v_{k-1}$ satisfying
\begin{equation*}
v_iv_{i+1}\subset H_i\; \text{for}\; 0\le i\le k-1 \pmod{k}.
\end{equation*}
A hypergraph is {\it linear}, also called nearly disjoint,
if every two edges meet in at most one vertex.
Let $C_{\ell}^{(3)}$ be the collection of 3-uniform Berge cycles of length $\ell$.

We write $\ex_r(n, \FF)$ ($\exrlin(n, \FF)$, resp.)
to denote the maximum number of hyperedges
in a $r$-uniform (and linear, resp.) hypergraph on $n$ vertices
that does not contain any member of $\FF$.
Gy{\H o}ri and Lemons~\cite{gy-nat-3unif} showed that
\begin{equation}\label{gy-lem3}
\ex\left(\left\lfloor \frac{n}{3}\right\rfloor , \left\lfloor \frac{n}{3}\right\rfloor , \CC_{2k}\right) \leq
\ex_3(n, C^{(3)}_{2k+1})< 4k^4n^{1+\frac{1}{k}} + 15k^4n + 10k^2n.
\end{equation}
The order of magnitude of the upper bound probably cannot be improved (as $k$ is fixed and $n\to \infty$).

Gy{\H o}ri and Lemons~\cite{gy-nat-kunif} extended their result
to $C^{(3)}_{2k}$-free 3-uniform hypergraphs (and also to $m$-uniform hypergraphs)
by showing  that the same lower bound as in~\eqref{gy-lem3}
holds for $\ex_3(n, C^{(3)}_{2k})$ and that $\ex_3(n, C^{(3)}_{2k})\leq c(k)n^{1+\frac{1}{k}}$.
The construction showing the lower bound in (\ref{gy-lem3}) is defined by considering a balanced bipartite
graph $G$ on $n/3+n/3$ vertices which is extremal not containing any members of $\CC_{2k}$.
A $3$-uniform $C^{(3)}_{2k}$-free hypergraph $\HH$ is formed by doubling each vertex in one of
the parts of $G$, thus turning each edge of $G$ to a hyperedge of $\HH$.
The number of hyperedges in $\HH$ is $e(G)=\ex(n/3, n/3, \CC_{2k})$.

In this paper, we make improvements on the bounds on $\ex_3(n, C^{(3)}_{2k+1})$ and
$\ex_3(n, C^{(3)}_{2k})$.
First, observe that trivially
\begin{equation}\label{eq8}
 t_{2k+1}(n)
 \leq \ex_3(n, C^{(3)}_{2k+1}).
   \end{equation}
(Consider the triple system defined by the triangles of a $C_{2k+1}$-free graph).
So (\ref{eq4}) gives a lower bound which (probably) improves the lower bound
 in (\ref{gy-lem3}) by a factor of $\Omega(k)$.

The aim of this paper is to improve the upper bound in (\ref{gy-lem3}) by a factor of
 (at least) $\Omega(k^2)$ and also to simplify the original proof.
In Section~\ref{3unif} we reduce the upper bound into three subproblems as follows.

\begin{thm}\label{main-thm-3unif} For $k\ge 2$ we have
\begin{eqnarray}\label{eq9}
\ex_3(n, C^{(3)}_{2k+1}) &\leq&  
                          t_{2k+1}(n)   + 4\ex(n, C_{2k}) + 12 \extrilin(n,C^{(3)}_{2k+1}), \\
\ex_3(n, C^{(3)}_{2k}) &\leq&  
                             t_{2k}(n)  + \ex(n, C_{2k}).\label{eq10}
\end{eqnarray}
\end{thm}
\noindent
The first and the third terms in (\ref{eq9}) are both lower bounds, and probably the middle term is the smallest one.
In Section~\ref{sec-linear} we estimate the third term. 
\begin{thm}\label{ck-linear}
 For $k\ge 2$ we have
\begin{equation}\label{eq11}
\extrilin(n, C^{(3)}_{2k+1}) \leq    2k n^{1 + 1/k} + 9kn.
\end{equation}
\end{thm}
\noindent
We were not able to relate the left hand side directly to $\ex(n,C_{2k})$.
In fact, just like in Gy{\H{o}}ri and Lemons' proof~\cite{gy-nat-3unif}, we reiterate a version
of the original proof of Bondy and Simonovits~\cite{Bondy-Simon} (as everybody else did
in~\cite{verst-2000}, \cite{pik-2011}, \cite{jiang}, and in~\cite{bukh}).
Our rendering is much simpler than~\cite{gy-nat-3unif}.
For the even case $\extrilin(n, C^{(3)}_{2k}) \leq  \ex(n,C_{2k})$ is obvious by
 selecting a pair from each hyperedge in a linear $C_{2k}$\! -free triple system.
We have no matching lower bound for $\extrilin(n, C^{(3)}_{\ell})$
 other than what follows from the random method.
Collier, Graber and Jiang~\cite{jiang} proved that
$\exrlin(n, C^{(r)}_{2k+1})\leq$ $\alpha _{k,r}$ $n^{1 + 1/k}$, but their
$\alpha_{k,r}$ is greater than $r(2k)^r$.
They find not only a Berge cycle but a \emph{linear cycle}, i.e., a cyclic list of triples
 such that consecutive sets intersect in exactly one element and nonconsecutive sets are disjoint.

Theorems~\ref{ct-tri},~\ref{main-thm-3unif} and~\ref{ck-linear} together with (\ref{pik}) imply
\[
\ex_3(n, C^{(3)}_{2k+1})\leq (9k^2+ 10k+5)n^{1 + 1/k} + O(k^2n)
\]
and $\ex_3(n, C^{(3)}_{2k}) \leq \frac{1}{3}(2k+9)(k-1)n^{1+1/k}+O(k^2 n)$.
Using (\ref{eq2}) one can lower the main coefficient to $O(k^{3/2}\sqrt{\log k})$.
If the smoothness conjectures concerning $\ex(n,C_{2k})$ and $\ex(n,n,\CC_{2k})$ hold, then
 the ratio of the upper bound (\ref{eq9}) and lower bound (\ref{eq8}) is of $O(a_k/b_k)$.

\section{Counting Triangles in $C_{2k}$-free and $C_{2k+1}$-free Graphs }\label{t2k+1}

We need the following classical result of Erd{\H o}s and Gallai~\cite{EG-path} on paths.
\begin{equation}\label{EG}
\ex(n,P_k) \leq \frac{k-2}{2}n.
\end{equation}

\begin{lem}\label{tg-count}
If $G$ is a $C_{\ell}$\! -free graph, then
$t(G)  \le  \frac{1}{3}(\ell - 3)e(G)$.
\end{lem}
\begin{proof}
For any vertex $x$, $t(x)$ equals to the number of edges induced by $N(x)$.
Therefore,
\begin{equation*}
t(G)  = \frac{1}{3} \sum_{x\in V(G)} t(x) = \frac{1}{3} \sum_{x\in V(G)} e(G[N(x)]).
\end{equation*}
The subgraph induced by $N(x)$ does not contain $P_{\ell - 1}$, because $G$ is $C_{\ell}$-free.
Therefore, by~(\ref{EG}), we have
\begin{equation*}
 e(G[N(x)]) \leq \frac{1}{2} (\ell - 3)\deg(x).
\end{equation*}
We obtain
\begin{equation*}
t(G)  \le \frac{1}{3} \sum_{x\in V(G)}  \frac{1}{2} (\ell - 3)\deg(x) =
 \frac{1}{3}(\ell - 3)e(G).  \qed
\end{equation*}

\end{proof}

Note that Lemma~\ref{tg-count} implies the upper bound (\ref{eq6}) for $t_{2k}(n)$.

\begin{proof}[Proof of Theorem~\ref{ct-tri}]
Let $G$ be a $C_{2k+1}$-free graph, $k\ge 2$, with the $n$ element vertex set $V$.
Let $\HH$ be the family of triangles in $G$.
Given any 3-partition (or 3-coloring) $\{V_1, V_2, V_3\}$ of $V$
 let $\HH(V_1, V_2, V_3)$ be the 3-partite induced subhypergraph of $\HH$ with these parts,
 i.e.,  $\HH(V_1, V_2, V_3):=\{ T\in \HH: |T\cap V_i|=1$ for all $1\leq i\leq 3\}$.
Standard averaging argument shows that there is a partition such that each color class
 $V_i$ with color $i$ has size $\lfloor (n+i-1)/3\rfloor$, $1\le i\le 3$, and
 the number of triples in $\HH':= \HH(V_1, V_2, V_3)$ is at least $2/9$'th of the number of triples in $\HH$.
So we have $|\HH|\leq (9/2)|\HH'|$.

Let $G'$ be the edges of $G$ contained in any triple from $\HH'$.
Since $t(G) =|\HH|$ and $t(G') =|\HH'|$, we have   $t(G) \leq (9/2)t(G')$.
From now on, our aim is to give an upper estimate for  $t(G')$.
Since $t(G') \leq \frac{1}{3}(2k-2)e(G')$ by Lemma~\ref{tg-count}, we have that
\[
t(G)  \leq  \frac{9}{2} t(G') \leq 3(k-1) e(G').
\]
To complete the proof of Theorem~\ref{ct-tri} we only need an appropriate upper bound on $e(G')$.

Let $G_{ij}$ be the bipartite subgraph of $G'$ induced by the vertex set
$V_i\cup V_j$, $1\le i<j\le 3$.
Assume that there exists a copy $L$ of $C_{2k}$ in $G_{ij}$ for some $i$ and $j$.
Let $x$ and $y$ be two adjacent vertices in $L$.
Since there exists a triangle in $G'$ with vertices $x, y, z$
for some  $z\in V_k$ ($k\ne i,j$),
there exists a copy of $C_{2k+1}$ in $G$ with the edge set $(E(L)-\{xy\})\cup \{xz,yz\}$, a contradiction.
Therefore, $G_{ij}$ is $C_{2k}$-free.
We obtain
\begin{equation*} 
e(G')=\sum_{1\le i<j\le 3} e(G_{i,j}) \leq 3 \ex(\lceil n/3 \rceil, \lceil n/3 \rceil, C_{2k}).
 \qed
\end{equation*}

\end{proof}

\section{$C^{(3)}_{\ell}$-free 3-uniform Hypergraphs}\label{3unif}

\begin{proof}[Proof of Theorem~\ref{main-thm-3unif}]~\\
For a pair of vertices $u$ and $v$, $\deg_\HH(u, v)$ (or just $\deg(u,v)$) denotes
the number of hyperedges of $\HH$ containing both $u$ and $v$.
\begin{prop}\label{key-prop}
Let $\HH$ be a $C^{(3)}_{\ell}$-free hypergraph, $\ell\ge 3$.
Let $G_2:=G_2(\HH)$ be the graph on the vertex set of $\HH$ such that
$E(G_2) := \{uv: \deg(u,v)\ge 2\}$.  Then, $G_2$ is $C_{\ell}$-free.
\end{prop}
\begin{proof}
Suppose, on the contrary, that $L$ is a cycle of length $\ell$ in $G_2$.
Let $\HH(e)$ be the set of triples from $\HH$ containing the pair $e$.
Suppose that $\ell\geq 4$, the case $\ell=3$ is trivial.
Then every triple $E\in \HH$ contains at most two edges from $E(L)$, but every
 $e\in E(L)$ is contained in at least two triples, Hall condition holds.
I.e., every $i$ edges of $E(L)$ (for $1\leq i\leq \ell$) are contained in at least $i$
triples.
So by Hall's theorem one can choose a distinct hyperedge from
$\HH(e)$ for each edge $e$ of $L$. These are forming a Berge cycle of length $\ell$, a contradiction.
\end{proof}

\noindent
{\it The upper bound on $\ex_3(n, C^{(3)}_{2k+1})$.}~\\
Let $\HH$ be a 3-uniform hypergraph that does not contain $C^{(3)}_{2k+1}$ as a subgraph.
Let $G_2$ be defined as in Proposition~\ref{key-prop}.
Then $G_2$ is $C_{2k+1}$-free.
Let $\HH_2$ be the collection of triples from $\HH$
having all the three pairs covered at least twice.
The edges of $\HH_2$ induce triangles in $G_2$, hence we have
\begin{equation}\label{bd-H2}
|\HH_2|\leq N(G_2; C_3)\leq t_{2k+1}(n). 
\end{equation}

Let $\HH_1$  be the set of triples $E$ from $\HH$ having a pair $P(E)$ such that
$P(E)$ is contained only in $E$.
Note that  $|\HH|=|\HH_1|+|\HH_2|$.
In the following, we find an upper bound for $|\HH_1|$ by defining further subfamilies $\HH_3, \dots, \HH_6$.

Color the vertices of $\HH_1$ randomly with two colors.
The probability that for an edge $E\in \HH_1$ the pair $P(E)$ gets the same color and
 the vertex $E\setminus P(E)$ has the opposite color is $1/4$.
This implies that there is a partition $V_1\cup V_2$ of $V(\HH)$ and a
 subfamily $\HH_3\subset \HH_1$ such that  $|\HH_3| \geq (1/4)|\HH_1|$ and
every edge $E$ of $\HH_3$ has two vertices in $V_i$ and
one vertex in $V_{3-i}$ for some $i\in \{1,2\}$ such that
$V_i\cap E = P(E)$.
Split $\HH_3$ into two subfamilies as follows.
\begin{multline}
\HH_4:= \{\{u,v,w\}\in \HH_3: P(E)=\{u,v\}\subset V_i,\;  w\in V_{3-i},\;\\
\max(\deg(w,u), \deg(w,v))\ge 3,\;
i\in \{1,2\} \}\notag
\end{multline}
and let $\HH_5:=\HH_3\setminus \HH_4$.

We claim that the graph $G_4$ consisting
of the pairs $P(E)$, $E\in \HH_4$, is $C_{2k}$-free.
Indeed, suppose, on the contrary, that $L=(v_1, \dots, v_{2k})$ is a cycle of $G_4$.
Since $G_4$ has no edge joining $V_1$ and $V_2$ we may suppose that $L\subset V_1$.
Consider the triples of $\HH_4$ containing the edges of $L$,
 $E_i:= \{ v_i, v_{i+1}, w_i\}$, $(1\leq i\leq 2k-1)$, and $E_{2k}:=\{ v_{2k}, v_1, w_{2k}\}$.
The vertices $w_1, \dots, w_{2k}$ are in $V_2$, so they are not on $L$.
Assume that $\deg(v_1, w_1)\geq 3$.
Then, there is a hyperedge $E_0=\{v_1, w_1, u\}\in \HH$ different from $E_1, \dots, E_{2k}$.
The hyperedges $\{ E_0,$ $E_1, $ $E_2, $ $\dots, $ $E_{2k}\}$ are containing
the consecutive pairs $\{v_1, w_1, v_2, \dots , v_{2k}\}$
in this cyclic order, so form a Berge cycle of length $2k+1$.
Thus,
\begin{equation}\label{bd-H4}
   |\HH_4|=e(G_4)\leq \ex(|V_1|, C_{2k}) + \ex(|V_2|, C_{2k}) \leq \ex(n, C_{2k}).
\end{equation}

Because the multiplicity of the pairs in any edge $E$ in $\HH_5$
is at most 2, one can use a greedy algorithm
to find a subfamily $\HH_6\subset \HH_5$
such that
$|\HH_6| \geq (1/3)|\HH_5|$, where $\HH_6$ is linear,
that is each vertex-pair is covered at most once by an edge of $\HH_6$.

Finally,
\begin{align}
|\HH| &=|\HH_1|+|\HH_2| \leq 4|\HH_3|+|\HH_2|=\notag\\
& =|\HH_2|+ 4|\HH_4| + 4|\HH_5| \leq |\HH_2|+ 4|\HH_4| + 12|\HH_6|. \notag
\end{align}
This with~\eqref{bd-H2},~\eqref{bd-H4},
and the linearity of $\HH_6$ completes the proof of (\ref{eq9}).

{\it The upper bound on $\ex_3(n, C^{(3)}_{2k})$.}~\\
Let $\HH$ be a 3-uniform hypergraph that does not contain $C^{(3)}_{2k}$ as a
subgraph.
Let $G_2$, $\HH_1$, $\HH_2$ be defined for $\HH$ as before.
By Proposition~\ref{key-prop}, $G_2$ is $C_{2k}$-free.
Hence,    $|\HH_2|\leq N(G_2; C_3)\leq t_{2k}(n)$. 

Recall that for each hyperedge $E$ in $\HH_1$, there exists a vertex-pair,
 $P(E)$, such that $P(E)$ is contained only in $E$ in $\HH$.
Let $G_1$ be the graph defined by its edge set as $E(G_1):= \{ P(E): E\in \HH_1 \}$.
We have that $|\HH_1|=e(G_1)$.
Since $G_1$ is obviously $C_{2k}$-free we get
\begin{equation*} 
|\HH| =|\HH_1|+|\HH_2| \leq   t_{2k}(n) +  \ex(n, C_{2k}). \qed
\end{equation*}

\end{proof}
\medskip

\section{$C^{(3)}_{\ell}$-free 3-uniform Linear Hypergraphs}\label{sec-linear}

A {\em theta graph} of order $\ell$, denoted by $\Theta_\ell$,
is a cycle $C_\ell$ with a chord, where  $\ell \geq 4$.
The following result was used implicitly in~\cite{Bondy-Simon}
 and is stated as a separate lemma
 in~\cite[Lemma 2]{verst-2000} and also used in~\cite{bukh} and~\cite{pik-2011}.
Let $F$ be a $\Theta$-graph of order  $\ell$ and $\ell > t\geq 2$.
Let $A\cup B$ be a partition of $V(F)$ with $A,B\neq \emptyset$ such that
  every path of length $t$ in $F$ that starts in $A$ necessarily ends in $A$.
Then  $F$ is bipartite with parts $A$ and $B$.
We need the following corollary, whose proof is left to the reader.
\begin{cor}\label{cor-11}
Let $F$ be a $\Theta$-graph of order  $\ell$, where
$\ell > t\geq 1$ and $t$ is an odd integer.
Let $A\cup B$ be a partition of $V(F)$, $A\neq \emptyset$ such that every path of length $t$ in $F$ that starts in $A$ necessarily ends
in $A$. Then  $A=V(F)$. \qed
\end{cor}

\vskip -2mm
We also use the following easy fact, which is used in~~\cite{Bondy-Simon},~\cite{bukh} and~\cite{pik-2011}, too.
If the $n$-vertex graph $G$ contains no $\Theta$-graph of order at least $\ell\geq 4$,
 then $e(G)  \le  (\ell - 2)n$.
 In other words
\begin{equation}
  \ex(n, \Theta_{\geq \ell} ) \leq (\ell -2)n.    \label{eq15}
  \end{equation}
\smallskip

\vskip -2mm
\begin{proof}[Proof of the upper bound on $\extrilin(n, C^{(3)}_{2k+1})$ in
Theorem~\ref{ck-linear}]~\\
Let $\HH$  be  a 3-uniform hypergraph on $n$ vertices such that no two
hyperedges meet in two vertices.
Suppose that $\HH$ contains no $C^{(3)}_{2k+1}$ and let
$\delta$ be the third of the the average degree.
We have $\sum_{v\in V(\HH)} \deg(v) = 3|\HH|=3\delta n$.
Then, there exists a subhypergraph $\HH'$ on  $n'$ vertices such that
the degree of each vertex of $\HH'$ is at least $\delta$.
Therefore, we may suppose that every degree of $\HH$ is at least $\delta$, and
also that $\delta \geq 11k$.

The mapping  $\pi : \HH \to \binom{[n]}{2}\cup {\emptyset}$  is called a
 {\it choice function} if $\pi(E)\subset E$ for each  $E\in \HH$.
There are $4^{|\HH|}$  such choice functions.
Let $\partial\HH$  be the set of vertex-pairs contained in the members of $\HH$
 and consider a coloring of $\partial\HH$, where the color of each pair is
 given by the single hyperedge of $\HH$ containing it.
We call a subgraph $G$ of $\partial\HH$ {\it multicolored},
 if all edges of $G$ have different colors under this coloring.
For a choice function $\pi$ on $\HH$, define the graph  $G_\pi$
 as the graph induced by the edge set
 $\{\pi(E): \pi(E)\neq \emptyset, E\in \HH\}$.
Because $\HH$ is a linear hypergraph, for two different hyperedges
 $E$ and $E'$ in $\HH$ we have $\pi(E) \neq \pi(E')$.
First, we consider the properties of arbitrary multicolored $G_\pi$,
 later we will define a special $\pi$.
Clearly, $G_\pi$ has no cycle $C_{2k+1}$.

\begin{lem}\label{lem15}
Let $T$ be a subtree (not necessarily spanning) in $G_\pi$, let $x\in V(T)$ be an arbitrary vertex,
 and let $V_i:= N_i(x)$ in $T$, the set of vertices of distance $i$ from $x$ in the tree $T$.
Consider $G_i:= G_\pi[V_i]$, the subgraph of  $G_\pi$  restricted to $V_i$.
Then $G_i$ has no $\Theta$-graph of order $2k$ or larger.
\end{lem}
\begin{cor}\label{cor-16}
$e(G_i)\leq (2k-2)|V_i|$ for $ 1\leq i \leq k$.
\end{cor}

\begin{proof}[Proof of Lemma~\ref{lem15}]
We use induction on $i$.
Since $V_0={x}$, and $V_1$ (more exactly $G_1$) contains no path of $2k$ vertices,
 it does not contain a $\Theta_{\geq 2k}$ either.
From now on, we may suppose that $i\geq 2$.

Suppose, on the contrary, that $F$ is a $\Theta$ subgraph of  $G_i$ of order
$\ell\geq 2k$, $i\geq 2$.
For arbitrary $y\in V_1$,  let $V_i(y)$ be the subset of
descendants of $y$ in $V_i$ in the tree $T$.
Consider the partition of $V_i$ defined as  $\{V_i(y): y\in V_1\}$.
There exists a $y_1\in V_1$ such that $A:=V(y_1)\cap V(F)\neq \emptyset$.

We claim that $F$ is contained in $V(y_1)$.
Note that there is no path $P(a,b)$ of $F$ (neither of $G_i$) of length $2k+1-2i$
 that starts in some vertex $a\in A \subset V_i(y_1)$ and
 ends in another vertex $b\in V_i\setminus V(y_1)$.
Otherwise, the $xy_1a$ and $xb$ paths on $T$ have only a single common vertex (namely $x$), have lengths $i$
 so together with $P(a,b)$ they form a $C_{2k+1}$ in $G_\pi$, a contradiction.
Therefore, every path of length $2k+1-2i$ in $F$, that starts in $A$ ends in $A$.
Corollary~\ref{cor-11} implies that $A=V(F)$, i.e., $V(F)\subset V(y_1)$.

To finish the proof of Lemma~\ref{lem15} simply use induction
 to the subtree $T_1$ of $T$ consisting of all descendants of $y_1$.
Then $N_{i-1}(y_1)$ in $T_1$ is exactly $V_i(y_1)$, so it does not contain
 any $\Theta_{\geq 2k}$.
\end{proof}

We say for two sets of sequences of integers
 $\alpha=$ $(a_1,$ $\dots,$ $a_k)$ and $\beta=$ $(b_1,$ $\dots,$ $b_k)$
 that $\alpha > \beta$, if there is an $i$ such that $a_i> b_i$ and $a_j=b_j$ for all $j<i$.
This is called the lexicographical ordering, and it is indeed a linear order.

We are ready to define a concrete $T$ and a choice function $\pi$.
Fix a vertex $x\in V(\HH)$ arbitrarily, let $V_0:=\{x\}$.
Consider all choice functions $\pi$ and all multicolored trees of $G_{\pi}$
 with root and center $x$ and radius at most $k$.
Let $T$ be such a tree for which the sequence of the neighborhood sizes
 ($|N_1(x)|,$ $\dots,$ $|N_k(x)|$) takes its maximum in the lexicographic order.
Since $\HH$ is linear we have $|N_1(x)|=\deg_\HH(x)$.
Recall that $N_i(x)$ is denoted by $V_i$, $0\leq i\leq k$.
Our aim is to prove that the sizes of the $|V_i|$'s increase rapidly as follows.

\begin{lem}\label{lem-17}
For $1\leq i\leq k-1$ we have
$|V_{i+1}|\geq \dfrac{\delta -7k}{2k} |V_i|$.
\end{lem}
This lemma completes the proof, because we obtain
$n\geq$ $|V_k|\geq$ $(\delta- 7k)^{k-1}$ $(2k)^{-k+1}$ $|V_1|$.
This and  $|V_1|=\deg_\HH(x)\geq \delta$ give $2kn^{1/k}+7k\geq \delta$.

\begin{proof}[Proof of Lemma~\ref{lem-17}]
Let  $\HH_i$ be the hyperedges of $\HH$ containing the edges of $T$
joining $V_i$ to $V_{i+1}$, $0 \leq i \leq k-1$, we have  $|\HH_i|=|V_{i+1}|$.
If $uvw=E \in \HH_i$  with $u\in V_i,$ $v\in V_{i+1},$
then $w \notin  V_j$ with $j<i$.
Otherwise, leaving out the edge  $uv$ from $T$ and joining  $wv$
results in a multicolored tree preceding $T$ in the lexicographic order.

Let $\BB_i$ be the set of hyperedges from
 $\HH \setminus$ $(\HH_0\cup$ $\HH_1\cup$ $\dots \cup$ $\HH_i)$
 meeting $V_i$, but not meeting $\cup_{j<i} V_j$, $0\leq i\leq k-1$.
We have $\BB_0=\emptyset$.
If  $E\in \BB_i$, then $E\subset V_i\cup V_{i+1}$.
Otherwise, if $u\in E\cap V_i$ and $v\in E\setminus (V_i\cup V_{i+1})$
then truncating our tree at $V_0\cup V_1\cup  \dots \cup V_{i+1}$
 and joining the edge $uv$ result in another tree lexicographically larger than $T$.

Let $\BB_i^{\alpha}$, $0\leq i\leq k-1$, be the set of those hyperedges from
 $\BB_i$, that meet $V_i$ exactly in $\alpha$ vertices, $\alpha =1$, $2$ or $3$.
The graph $G_i$, for $1\le i \le k-1$, is defined on the vertex set $V_i$ as follows.
It contains exactly one vertex-pair from each member of
$\BB_i^3$ and the pairs $E\cap V_i$ for $E\in \BB_i^2 \cup \BB_{i-1}^1$.
For $i=k$, the edge set of $G_k$ consists only of the sets $\{E\cap V_k: E\in \BB_{k-1}^1\}$, since
  $\BB_k$ is undefined.
The graph $G_\pi$ consisting of the edges of $T$ and the $G_i$'s,
$1\leq i\leq k$, is a multicolored subgraph.
So Corollary~\ref{cor-16} implies that
\begin{equation}\label{eqn-20}
e(G_i)\leq (2k-2) |V_i|.
\end{equation}

Consider the $\HH$-degrees of the elements of $V_i$, $(1\le i\le k-1)$.
Their total sum is at least $\delta|V_i|$.
Obviously,
\[
\sum_{v\in V_i} \deg_\HH(v) = \sum_{E\in \HH} |E\cap V_i|.
\]
The edges of $\HH$ meeting $V_i$ belong to some $\HH_j$, $j\leq i$,
or to $\BB_{i-1}\cup \BB_i$.
An edge $E\in \HH_j$ can meet $V_i$ in at least two elements,
only if $j$ is equal to $i-1$ or $i$.
We obtain for $1\leq i\leq k-1$
\begin{multline}\notag
\delta |V_i|\leq \sum_{v\in V_i} \deg(\HH)(v) = \sum_{E\in \HH}
|E\cap V_i| \\ \leq
\left( \sum_{0\le j\le i-2} |\HH_j| \right) +
2|\HH_{i-1}|+ 2|\HH_i|+    
|\BB_{i-1}^2| + 2|\BB_{i-1}^1| + 3|\BB_i^3| + 2|\BB_i^2| + |\BB_i^1|.
\end{multline}
Inequality~\eqref{eqn-20} implies that
\begin{align}
|\BB_{i-1}^2|&\leq e(G_{i-1}) \leq (2k-2)|V_{i-1}|,\notag\\
2|\BB_{i-1}^1| + 3|\BB_i^3| + 2|\BB_i^2|&\leq
              3(|\BB_{i-1}^1| + |\BB_i^3|+|\BB_i^2|)=3e(G_i) \leq (6k-6)|V_i|, \notag\\
|\BB_i^1|&\leq e(G_{i+1}) \leq (2k-2)|V_{i+1}|. \notag
\end{align}
Using these inequalities and the fact that $|\HH_j|=|V_{j+1}|$ we obtain that
\begin{equation*} 
    \delta |V_i| \leq \left(\sum_{1\leq j\leq  i-1} |V_j|\right) + 2|V_i| + 2|V_{i+1}|+
        (2k-2)|V_{i-1}| + (6k-6)|V_i| + (2k-2)|V_{i+1}|.
\end{equation*} 
By rearranging we have
\begin{equation}\label{eqn-17}
(\delta -(6k-4)) |V_i| \leq \left(\sum_{1\leq j\leq i-1} |V_j|\right) +  (2k-2)|V_{i-1}| +
2k|V_{i+1}|.
\end{equation}
For $i=1$  the fact that $\BB_0=\emptyset$
implies the slightly stronger $(\delta -(6k-4)) |V_1| \leq 2k|V_2|$.
So Lemma~\ref{lem-17} holds for $i=1$.
For larger $i$ we use induction and~\eqref{eqn-17} to prove first
 that $2|V_i|\leq |V_{i+1}|$ for all $i<k$ and then the sharper inequality of Lemma~\ref{lem-17}.
\end{proof}
\end{proof}



\bibliographystyle{amsplain}
\bibliography{count-triangle}

\providecommand{\bysame}{\leavevmode\hbox to3em{\hrulefill}\thinspace}
\providecommand{\MR}{\relax\ifhmode\unskip\space\fi MR }
\providecommand{\MRhref}[2]{%
  \href{http://www.ams.org/mathscinet-getitem?mr=#1}{#2}
}
\providecommand{\href}[2]{#2}
\begin{thebibliography}{10}

\bibitem{alon-shik}
N.~Alon and C.~Shikhelman, \emph{Triangles in {$H$}-free graphs},
   \newline
   http://arxiv.org/abs/1409.4192.

\bibitem{bol-gy-2008}
B.~Bollob{\'a}s and E.~Gy{\H{o}}ri, \emph{Pentagons vs. triangles}, Discrete
  Math. \textbf{308} (2008), no.~19, 4332--4336.

\bibitem{Bondy-Simon}
J.~A. Bondy and M.~Simonovits, \emph{Cycles of even length in graphs}, J.
  Combin. Theory Ser. B \textbf{16} (1974), 97--105.

\bibitem{bukh}
B.~Bukh and Z.~Jiang, \emph{A bound on the number of edges in graphs without an even cycle},
  \newline
  http://arxiv.org/abs/1403.1601.

\bibitem{jiang}
C.~Collier, C.~N. Graber, and T.~Jiang, \emph{Linear tur\'an numbers of
  $r$-uniform linear cycles and related ramsey numbers},
  http://arxiv.org/abs/1404.5015.

\bibitem{erdos}
P.~Erd{\H{o}}s, \emph{On some problems in graph theory, combinatorial analysis
  and combinatorial number theory}, Graph theory and combinatorics
  ({C}ambridge, 1983), Academic Press, London, 1984, pp.~1--17.

\bibitem{EG-path}
P.~Erd{\H{o}}s and T.~Gallai, \emph{On maximal paths and circuits of graphs},
  Acta Math. Acad. Sci. Hungar \textbf{10} (1959), 337--356.

\bibitem{Grz}
A.~Grzesik, \emph{On the maximum number of five-cycles in a triangle-free
  graph}, J. Combin. Theory Ser. B \textbf{102} (2012), no.~5, 1061--1066.

\bibitem{gyC5}
E.~Gy{\H{o}}ri, \emph{On the number of $C_5$'s in a triangle-free graph},
Combinatorica \textbf{9} (1989), 101--102.

\bibitem{gy-nat-3unif}
E.~Gy{\H{o}}ri and N.~Lemons, \emph{3-uniform hypergraphs avoiding a given odd
  cycle}, Combinatorica \textbf{32} (2012), no.~2, 187--203.

\bibitem{gy-nat-kunif}
E.~Gy{\H o}ri and N.~Lemons, \emph{Hypergraphs with no cycle of a given
  length}, Combin. Probab. Comput. \textbf{21} (2012), no.~1-2, 193--201.

\bibitem{gy-li-2009}
E.~Gy{\H o}ri and H.~Li, \emph{The maximum number of triangles in
  {$C_{2k+1}$}-free graphs}, Combin. Probab. Comput. \textbf{21} (2012),
  no.~1-2, 187--191.

\bibitem{Razb-pentagon}
H.~Hatami, J.~Hladk{\'y}, D.~Kr{\'a}l, S.~Norine, and A.~Razborov, \emph{On the
  number of pentagons in triangle-free graphs}, J. Combin. Theory Ser. A
  \textbf{120} (2013), no.~3, 722--732.

\bibitem{KovSosTur}
T. K\H{o}v\'ari, V. T. S\'os, and P. Tur\'an:
\emph{On a problem of K. Zarankiewicz},
Colloq. Math. \textbf{3} (1954), 50--57.

\bibitem{pik-2011}
O.~Pikhurko, \emph{A note on the {T}ur\'an function of even cycles}, Proc.
  Amer. Math. Soc. \textbf{140} (2012), no.~11, 3687--3692.

\bibitem{verst-2000}
J.~Verstra{\"e}te, \emph{On arithmetic progressions of cycle lengths in
  graphs}, Combin. Probab. Comput. \textbf{9} (2000), no.~4, 369--373.

\end{thebibliography}
\end{document}